\documentclass[11pt]{article}

\usepackage{amssymb, amsmath, fullpage, amsthm}

\newtheorem{theorem}{Theorem}[section]
\newtheorem{lemma}[theorem]{Lemma}
\newtheorem{proposition}[theorem]{Proposition}
\newtheorem{corollary}[theorem]{Corollary}

\newtheorem{definition}[theorem]{Definition}
\newtheorem{example}[theorem]{Example}
\newtheorem{continuation}[theorem]{Continuation of Example}
\newtheorem{remark}[theorem]{Remark}

 \numberwithin{equation}{section}

\newcommand{\hz}{\widehat{0}}
\newcommand{\ho}{\widehat{1}}

\newcommand{\av}{{\bf a}}
\newcommand{\bv}{{\bf b}}
\newcommand{\cv}{{\bf c}}
\newcommand{\dv}{{\bf d}}
\newcommand{\ab}{\av\bv}
\newcommand{\cd}{\cv\dv}
\newcommand{\zab}{{\mathbb Z}\langle\av,\bv\rangle}
\newcommand{\zcd}{{\mathbb Z}\langle\cv,\dv\rangle}

\newcommand{\maa}{\av\av}
\newcommand{\mba}{\bv\av}
\newcommand{\mab}{\av\bv}
\newcommand{\mbb}{\bv\bv}

\newcommand{\mccc}{\cv^{3}}

\newcommand{\mdc}{\dv\cv}

\newcommand{\tensor}{\otimes}
\newcommand{\capdots}{\cap \cdots \cap}

\newcommand{\subseteqdots}{\subseteq \cdots \subseteq}

\DeclareMathOperator{\link}{link}

\DeclareMathOperator{\wt}{wt}

\newcommand{\Nnn}{{\mathbb N}}
\newcommand{\Rrr}{{\mathbb R}}
\newcommand{\Zzz}{{\mathbb Z}}

\newcommand{\PPP}{\widehat{P}}

\newcommand{\zetabar}{\overline{\zeta}}
\newcommand{\mubar}{\overline{\mu}}
\newcommand{\fbar}{\overline{f}}
\newcommand{\hhbar}{\overline{h}}

\begin{document}

\author{Richard Ehrenborg\thanks{Corresponding author:
Department of Mathematics,
University of Kentucky,
Lexington, KY 40506-0027,
USA,
{\tt jrge@ms.uky.edu},
phone +1 (859) 257-4090,
fax +1 (859)  257-4078.}
and Margaret Readdy\thanks{Department of Mathematics,
University of Kentucky,
Lexington, KY 40506-0027,
USA,
{\tt readdy@ms.uky.edu}.}
}

\title{Manifold arrangements}

\date{}

\maketitle

\begin{abstract}
We determine the $\cd$-index of the induced 
subdivision
arising from a manifold arrangement.
This generalizes earlier results in several directions:
($i$) One can work with manifolds other than the $n$-sphere
and $n$-torus,
($ii$) the induced subdivision is a Whitney stratification,
and 
($iii$) the submanifolds in the arrangement are no
longer required
to be codimension one.

\vspace*{2 mm}


\vspace*{2 mm}

\end{abstract}

\section{Introduction}

In the paper~\cite{Ehrenborg_Goresky_Readdy}
Ehrenborg, Goresky and Readdy extend the theory
of flag enumeration in polytopes and regular CW-complexes
to Whitney stratified manifolds.  
Their key insight is to replace flag enumeration with
Euler flag enumeration, that is,
a chain of strata is weighted by the Euler characteristic
of each link in the chain;
see Theorem~\ref{theorem_Whitney_Eulerian}.
The classical results for the generalized Dehn--Sommerville
relations and
the $\cd$-index ~\cite{Bayer_Billera,Bayer_Klapper, Stanley} 
carry over to this setting.

The $\cd$-index of a polytope, and more generally,
an Eulerian poset, is
a minimal encoding of the flag vector
having coalgebraic structure
which reflects the geometry of the polytope~\cite{Ehrenborg_Readdy}.
It is known that the coefficients of the
$\cd$-index of polytopes, spherically-shellable
posets and Gorenstein* posets are nonnegative~\cite{Karu,Stanley}.
Hence, inequalities among the $\cd$-coefficients imply inequalities
among the flag vector entries of the original objects.
For Whitney stratified manifolds, 
the coefficients of the $\cd$-index
are no longer restricted to being
nonnegative~\cite[Example~6.15]{Ehrenborg_Goresky_Readdy}.
This broadens the research program of understanding
flag vectors of polytopes to that of manifolds.
See~\cite{Ehrenborg_lifting} for the best currently known
inequalities for flag vectors of polytopes.

One would like to understand the combinatorics of naturally
occurring Whitney stratified spaces.  One such example
is that of manifold arrangements. These arrangements are motivated by
subspace arrangements in Euclidean space.
Classically Goresky and MacPherson determined the
cohomology of the complement of subspace arrangements
using intersection homology~\cite{Goresky_MacPherson}.
The stable homotopy type of the complement
was studied by Ziegler and \v{Z}ivaljevi\'{c}~\cite{Ziegler_Zivaljevic}.
Arrangements of submanifolds and subvarieties
have been considered in connection with
blowups in algebraic geometry;   see
for instance
MacPherson and Procesi's work on conical
stratification~\cite{MacPherson_Procesi},
and Li's work on arrangements of 
subvarieties~\cite{Li}.
Face enumeration issues were examined by
Zaslavsky~\cite{Zaslavsky_2} and Swartz~\cite{Swartz}.
See also~\cite{Forge_Zaslavsky}
where Forge and Zaslavsky extend the notion of hyperplane
to topological hyperplanes.

The case of toric arrangements,
that is, a collection of subtori inside
an $n$-dimensional torus~$T^{n}$, was studied by
Novik, Postnikov and Sturmfels~\cite{Novik_Postnikov_Sturmfels}
in reference to minimal cellular resolutions of unimodular matroids.
Other results for toric arrangements include that of 
De Concini and Procesi~\cite{De_Concini_Procesi},
who computed the 
cohomology of the complement 
and related this to polytopal lattice point enumeration,
and
D'Antonio and Delucchi~\cite{dAntonio_Delucchi},
who
considered the homotopy type 
and the fundamental group of the complement.

Billera, Ehrenborg and Readdy studied 
oriented matroids and the  lattice of 
regions~\cite{Billera_Ehrenborg_Readdy}.
The $\cd$-index of this lattice only depends upon the
flag $f$-vector of the intersection lattice, which is
a smaller poset.
Their work shows how to determine the $\cd$-index
of induced subdivisions of the sphere $S^{n}$.
Ehrenborg, Readdy and Slone 
considered toric arrangements that induce regular
subdivisions of the torus $T^{n}$~\cite{Ehrenborg_Readdy_Slone}.
Yet again, the associated $\cd$-index depends only upon the
flag $f$-vector of the intersection poset.

In this paper we consider arrangements of manifolds
and the subdivisions they induce.
In the manifold setting the computation of the $\cd$-index
now depends upon the intersection poset
and the Euler characteristic  of the elements of this quasi-graded
poset. 
This extends the earlier studied spherical and toric
arrangements~\cite{Billera_Ehrenborg_Readdy,Ehrenborg_Readdy_Slone}.

In order to obtain this generalization,
we first review the notion of a quasi-graded poset.
See Section~\ref{section_preliminaries}.
This allows us to
work with intersection posets that are not necessarily graded.
A short discussion of the properties of the Euler characteristic with
compact support and its relation to the Euler characteristic
is included.
We then introduce manifold arrangements
in Section~\ref{section_manifold_arrangements}.
The intersection poset of an arrangement is defined.
This notion is not unique. However, this gives us the advantage of
choosing the most suitable poset for calculations.
In Section~\ref{section_complement}
the Euler characteristic of
the complement is computed for these arrangements.
This is a manifold analogue
of the classical result concerning the number of regions
of a hyperplane arrangement~\cite{Zaslavsky_1}.
In Section~\ref{section_cd}
we review the notions of Eulerian
quasi-graded posets, the $\cd$-index
and Whitney stratified spaces.

In the oriented matroid setting
the coalgebraic structure of flag vector enumeration
was essential in developing the results.
In Section~\ref{section_coalgebra}
we describe the underlying coalgebraic structure
in the more general quasi-graded poset setting
and summarize the essential operators
from~\cite{Billera_Ehrenborg_Readdy}.
Using these operators, we define the
operator ${\cal G}$ that will be key
to the main result.

In Section~\ref{section_induced}
we state and prove the main result;
see Theorem~\ref{theorem_main}.
The previous proof techniques 
for studying subdivisions induced by
oriented matroids and toric arrangements
depended upon finding a natural map from the face poset
of the given subdivision to the intersection poset and studying
the inverse image of a chain under this map.
See the proofs of~\cite[Theorem~4.5]{Bayer_Sturmfels},
\cite[Theorem~3.1]{Billera_Ehrenborg_Readdy}
and~\cite[Theorems~3.12 and~4.10]{Ehrenborg_Readdy_Slone}.
In the more general setting of manifold arrangements
we can now avoid this step
by forming another
Whitney stratification having the same $\cd$-index;
see Proposition~\ref{proposition_T_Q}.
Namely, we can choose each strata to be a submanifold
in the intersection poset without those points included in
smaller submanifolds.
(It is customary to refer to a single stratum by the plural strata.)
In the classical case of hyperplane arrangements
this gives a stratification into disconnected strata.

Finally
in Section~\ref{section_sphere_torus}
we revisit two important cases studied earlier:
spherical and toric arrangements.
These arrangements have
the property that the Euler characteristic of an element
of dimension~$k$
in the intersection lattice only depends upon $k$,
that is, the Euler characteristic is given by 
$1+(-1)^{k}$, respectively the Kronecker delta $\delta_{k,0}$.
In both of these cases Theorem~\ref{theorem_main}
reduces to a result which only depends on the intersection poset.
The original work for spherical and toric
arrangements required the induced subdivision to yield a
regular subdivision on the sphere~$S^n$, respectively, the torus
$T^n$~\cite{Billera_Ehrenborg_Readdy,Ehrenborg_Readdy_Slone}.
This regularity condition is no longer necessary in the arena
of Whitney stratified spaces.

An illuminating sample of this theory is
to consider a complete flag in $n$-dimensional
Euclidean space. Intersecting this flag with
the $(n-1)$-dimensional unit sphere $S^{n-1}$ gives
a (nested) arrangement of spheres, one of each dimension.
The intersection poset is a chain of rank $n$.
The induced subdivision of the sphere consists of
two cells of each dimension $i$, $0 \leq i \leq n-1$.
The face poset is the butterfly poset of rank $n+1$.
It is straightforward to see that the classical
Billera--Ehrenborg--Readdy formula holds in this case;
see Example~\ref{example_complete_flag}.

We end with some open questions in the concluding remarks.

\section{Preliminaries}
\label{section_preliminaries}

A {\em quasi-graded poset} is a triplet $(P, \rho, \zetabar)$
where
\begin{itemize}
\item[(i)]
$P$ is a finite poset with
a minimal element $\hz$ and maximal element $\ho$,
\item[(ii)]
$\rho$ is a function from $P$ to $\Nnn$ such that
$\rho(\hz) = 0$ and $x < y$ implies $\rho(x) < \rho(y)$, and
\item[(iii)] $\zetabar$ is a function in the incidence algebra of $P$
such that for all $x$ in $P$ we have $\zetabar(x,x) = 1$.
\end{itemize}
The notion of quasi-graded poset is due to the authors and Goresky.
See~\cite{Ehrenborg_Goresky_Readdy} for further details
and see~\cite{EC1} for standard poset terminology.
In this paper we will assume $\zetabar$ is integer-valued,
though in general this is not necessary.

Condition (iii) guarantees that $\zetabar$ is invertible in the incidence algebra of $P$. Let $\mubar$ denote the inverse of $\zetabar$.
Observe that when the weighted zeta function $\zetabar$ is
the classical zeta function $\zeta$,
that is, $\zeta(x,y) = 1$ for all $x \leq y$,
the function~$\mubar$ is the M\"obius function $\mu$.

Recall that a {\em subspace arrangement}
is a collection $\{V_{i}\}_{i=1}^{m}$
of subspaces in $n$-dimensional Euclidean space.
We allow a subspace $V_{i}$ to be a proper subspace
of another subspace $V_{j}$ of the arrangement.
However, if a subspace $V_{i}$ is the intersection
of other subspaces in the arrangement,
that is,
$V_{i} = \bigcap_{j \in J} V_{j}$
for $J \subseteq \{1,2, \ldots, m\} - \{i\}$,
the subspace $V_{i}$ is redundant
for our purposes, and can be removed.

The {\em intersection lattice} of a subspace arrangement forms
a quasi-graded poset $(P,\rho,\zetabar)$ where
$P$ is the collection of all intersections of subspaces
ordered by reverse inclusion.
The minimal element is
the ambient space $\Rrr^{n}$ and the
maximal element is the intersection
$V_{1} \capdots V_{m}$.
The rank function $\rho$ is given by the codimension,
that is, $\rho(x) = n - \dim(x)$.
Finally, we let the weighted zeta function be given by
the classical zeta function $\zeta$,
where $\zeta(x,y) = 1$ for $x \leq y$.

One of the topological tools
we will need is the {\em Euler characteristic
with compact support}~$\chi_{c}$. For a reference, see
the article by Gusein-Zade~\cite{Gusein-Zade}.
We review two essential properties.
First, the Euler characteristic with compact support
is a {\em valuation}, that is, it satisfies
\begin{equation}
\chi_{c}(A) + \chi_{c}(B)
    = 
\chi_{c}(A \cap B) + \chi_{c}(A \cup B)  ,
\end{equation}
where the two sets $A$ and $B$ are 
formed by finite intersections, finite unions and complements
of locally closed sets.
(A {\em locally closed set}
is the intersection of an open set and a closed
set.)
The second key property relates 
the usual Euler characteristic~$\chi$ with~$\chi_{c}$.
\begin{proposition}
For an $n$-dimensional manifold $M$ that is
not necessarily compact, the Euler characteristic $\chi(M)$
and the Euler characteristic with compact support $\chi_{c}(M)$
satisfy the relation
\begin{equation}
   \chi_{c}(M)  = (-1)^{n} \cdot \chi(M)  .  
\end{equation}
\label{proposition_sign}
\end{proposition}
\begin{proof}
By Poincar\'e duality
for non-compact manifolds $M$ we have
\begin{equation*}
H_{c}^{i}(M; \Zzz_{2}) \cong H_{n-i}(M; \Zzz_{2})  .
\label{equation_Poincare}
\end{equation*}
The result follows by
taking dimension of this isomorphism,
multiplying by the sign $(-1)^{i}$ and
summing over all~$i$.
\end{proof}

Observe that by computing the
(co)-homology groups over the field of two elements,
we also cover the case when the manifold is non-orientable.

\section{Manifold arrangements}
\label{section_manifold_arrangements}

Given an $n$-dimensional compact manifold $M$
without boundary
and a collection of submanifolds
$\{N_{i}\}_{i=1}^{m}$ of $M$
each without boundary,
we call this collection a {\em manifold arrangement}
if it satisfies
Bott's~\cite[Section~5]{Bott}
{\em clean intersection property},
defined as follows.
For every point $p$ in the manifold~$M$,
there exist
(i) a neighborhood $U$ of $p$,
(ii) a neighborhood $W$ in $\Rrr^{n}$ of the origin,
(iii) a subspace arrangement $\{V_{i}\}_{i=1}^{k}$ in $\Rrr^{n}$,
and
(iv) a diffeomorphism $\phi : U \longrightarrow W$
such that
the point $p$ is mapped to the origin
and
the collection of manifolds
restricted to the neighborhood~$U$,
that is,
$\{N_{i} \cap U\}_{i=1}^{m}$
is mapped to the restriction of the subspace arrangement 
$\{V_{i} \cap W\}_{i=1}^{k}$.

Similar to the setup for subspace arrangements we
allow a manifold $N_{i}$ to be a proper submanifold
of another manifold $N_{j}$ in the arrangement;
see Examples~\ref{example_complete_flag}
and~\ref{example_toric_complete_flag}.
A manifold $N_{i}$ is redundant
if it is the intersection
of other manifolds in the arrangement,
that is,
$N_{i} = \bigcap_{j \in J} N_{j}$
for $J \subseteq \{1,2, \ldots, m\} - \{i\}$.

\begin{example}
{\rm
Let $M$ be the sphere $x^{2} + y^{2} + z^{2} = 2$,
and
let $\{N_{1},N_{2}\}$ consist of the two circles
$x^{2} + y^{2} = 1$, $z=1$;
and
$x=1$, $y^{2} + z^{2} = 1$.
Observe that the line $x=z=1$
is tangent to both circles at their point
of intersection $p = (1,0,1)$.
Hence at the point $p$
the clean intersection property
is not satisfied
so that $\{N_{1},N_{2}\}$ is not a manifold arrangement.
}
\end{example}

\begin{proposition}
An intersection $\bigcap_{i \in I} N_{i}$
where $I \subseteq \{1,\ldots,m\}$
in a manifold arrangement
$\{N_{i}\}_{i=1}^{m}$ of a compact manifold $M$
consists of a disjoint union of a finite number of
connected manifolds.
\label{proposition_components}
\end{proposition}
\begin{proof}
Assume that the intersection
$\bigcap_{i \in I} N_{i}$
consists of an infinite number
of connected components.
Pick a point $p_{j}$ from each connected component.
Since $M$ is compact the sequence $\{p_{j}\}_{j \geq 1}$
has a convergent subsequence.
Let $p$ be the limit point of this subsequence.
Observe now that the clean intersection property does
not hold at $p$, contradicting the assumption of the
existence of an infinite number of components.

The clean intersection property implies that the
neighborhood of a point in the intersection
$\bigcap_{i \in I} N_{i}$ is a relatively open ball,
that is, each connected component is a manifold.
\end{proof}

\begin{example}
{\rm
To illustrate why compactness is a necessary condition,
consider the two curves $y = \sin(x)$ and $y = -\sin(x)$
in the plane ${\mathbb R}^{2}$.
They intersect in an infinite number of points.
}
\end{example}

The connected components in
Proposition~\ref{proposition_components}
could be manifolds of different dimensions.
We illustrate this behavior
in Example~\ref{example_different_dimensions}.
\begin{example}
{\rm
Let $A$, $B$, $C$ and $D$
be the following four $2$-dimensional spheres in ${\mathbb R}^{4}$:
\begin{align*}
A & =  \{(x,y,z,w) \: : \: x=0, (y-1)^{2} + z^{2} + w^{2} = 1\}, \\
B & =  \{(x,y,z,w) \: : \: x^{2} + y^{2} + (z-1)^{2} = 1, w = 0\}, \\
C & =  \{(x,y,z,w) \: : \: (x-9)^{2} + y^{2} + z^{2} = 2, w = 0\}, \\
D & =  \{(x,y,z,w) \: : \: (x-11)^{2} + y^{2} + z^{2} = 2, w = 0\}.
\end{align*}
The spheres $A$ and $B$ intersect in two points,
the spheres $C$ and $D$ intersect in a circle,
and there are no other intersections.
Now construct the connected sums $A\#C$ and $B\#D$ by attaching
disjoint tubes.
We obtain two spheres
$A\#C$ and $B\#D$ which intersect in two points and a circle.
Finally, take the
one-point compactification
of ${\mathbb R}^{4}$ to obtain an arrangement in the
four-dimensional sphere~$S^{4}$.
}
\label{example_different_dimensions}
\end{example}

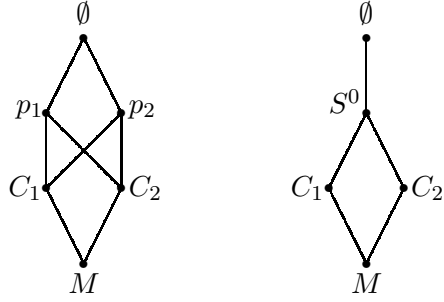
\begin{figure}
\setlength{\unitlength}{1mm}
\begin{center}
\begin{picture}(10,30)(0,0)

\qbezier(5,0)(5,0)(0,10)
\qbezier(5,0)(5,0)(10,10)

\qbezier(0,10)(0,10)(0,20)
\qbezier(0,10)(0,10)(10,20)
\qbezier(10,10)(10,10)(0,20)
\qbezier(10,10)(10,10)(10,20)

\qbezier(0,20)(0,20)(5,30)
\qbezier(10,20)(10,20)(5,30)

\put(5,0){\circle*{1}}
\put(0,10){\circle*{1}}
\put(10,10){\circle*{1}}
\put(0,20){\circle*{1}}
\put(10,20){\circle*{1}}
\put(5,30){\circle*{1}}

\put(3,-4){$M$}
\put(-5,9){$C_{1}$}
\put(11,9){$C_{2}$}
\put(-4,20){$p_{1}$}
\put(11,20){$p_{2}$}
\put(4,32){$\emptyset$}

\end{picture}
\hspace*{25 mm}
\begin{picture}(10,30)(0,0)

\qbezier(5,0)(5,0)(0,10)
\qbezier(5,0)(5,0)(10,10)

\qbezier(0,10)(0,10)(5,20)
\qbezier(10,10)(10,10)(5,20)

\qbezier(5,20)(5,20)(5,30)

\put(5,0){\circle*{1}}
\put(0,10){\circle*{1}}
\put(10,10){\circle*{1}}
\put(5,20){\circle*{1}}
\put(5,30){\circle*{1}}

\put(3,-4){$M$}
\put(-5,9){$C_{1}$}
\put(11,9){$C_{2}$}
\put(0,20){$S^{0}$}
\put(4,32){$\emptyset$}

\end{picture}
\end{center}
\caption{The two possible intersection posets
for the arrangement in
Example~\ref{example_two_possibilities}.}
\label{figure_two_possibilities}
\end{figure}

We now introduce the notion of intersection poset.
Depending on particular circumstances
there could be several suitable intersection posets
for a given arrangement. 
\begin{example}
{\rm
Let $M$ be the sphere $x^{2} + y^{2} + z^{2} = 1$.
Let $\{C_{1},C_{2}\}$
be the arrangement
consisting of the two circles
$x^{2} + y^{2} = 1$, $z=0$;
and
$x=0$, $y^{2} + z^{2} = 1$,
which intersect in two points
$p_{1}$ and~$p_{2}$.
We can either view these two points
as separate elements in an intersection poset
or as one zero-dimensional sphere $S^{0}$.
We will see that both views are useful. The
two possible intersection posets are displayed in
Figure~\ref{figure_two_possibilities}.
}
\label{example_two_possibilities}
\end{example}

\begin{definition}
An intersection poset $P$ of a manifold arrangement
$\{N_{i}\}_{i=1}^{m}$ of a compact manifold~$M$
is a poset whose elements are ordered by reverse inclusion
that satisfies:
\begin{itemize}
\item[(i)]
The empty set is an element of $P$.

\item[(ii)]
Each non-empty element of $P$ is a disjoint union
of connected components, all of the same dimension,
of a non-empty intersection 
$\bigcap_{i \in I} N_{i}$
where $I \subseteq \{1,\ldots,m\}$.

\item[(iii)]
Given a connected component $C$ of 
an intersection $\bigcap_{i \in I} N_{i}$,
there exists exactly one element of $P$
having $C$ as one of its connected components.
\end{itemize}
Conditions (ii) and (iii) imply that
each intersection can be written uniquely as a disjoint union
of non-empty elements of $P$.
\begin{itemize}
\item[(iv)]
Let $I \subseteq J$ be two index sets.
Then we have unique subsets $A$ and $B$ of $P - \{\emptyset\}$
such that
$$   \bigsqcup_{x \in A} x
     =
       \bigcap_{i \in I} N_{i}
     \supseteq
       \bigcap_{j \in J} N_{j}
     =
     \bigsqcup_{y \in B} y  .  $$
If $x \in A$ and $y \in B$ intersect non-trivially
then the element $x$ contains the element $y$,
that is,
$$   x \cap y \neq \emptyset
  \Longrightarrow
       x \supseteq y    .   $$
\end{itemize}
\end{definition}
The condition that the elements of $P$ consist of
manifolds all of the same dimension ensures that
the dimension of a non-empty element of $P$ is well-defined.
Also note that condition $(iv)$ mimics the condition of the frontier
for stratified spaces;
see equation~\eqref{equation_condition_of_the_frontier}.

\begin{example}
{\rm
Let $M$ be a compact manifold of dimension greater than one
and let $N_{1}$ and $N_{2}$ be two one-dimensional submanifolds
of $M$, that is, $N_{1}$ and $N_{2}$ are closed curves.
Assume that 
$N_{1}$ and $N_{2}$
intersect in $k$ points.
Then the number of possible intersection posets
of the manifold arrangement $\{N_{1},N_{2}\}$
is given by the $k$th Bell number, that is, the number of
set partitions of a $k$-element set.
}
\end{example}

As an example of an intersection poset, we may take
the elements to consist of all connected components of the
non-empty intersections. This is the approach taken
in the paper~\cite{Ehrenborg_Readdy_Slone} when
studying toric arrangements. However, this does not
work for spherical arrangements since the zero-dimensional
sphere consists of two points and hence is disconnected.
See Section~\ref{section_sphere_torus} for
further discussion regarding these two special cases.

\begin{example}
{\rm
Let $\{N_{i}\}_{i=1}^{m}$ be a manifold arrangement
of a manifold $M$
such that each intersection
$\bigcap_{i \in I} N_{i}$
is {\em pure}, that is,
each component of 
$\bigcap_{i \in I} N_{i}$
has the same dimension.
Then as an intersection poset
we may choose
$$   L
    =
       \left\{ \bigcap_{i \in I} N_{i}
            \:\: : \:\:
                I \subseteq \{1, \ldots, m\} \right\}
     \cup
       \{ \emptyset \}   . $$
Observe that this intersection poset is indeed a lattice
and hence it is called the {\em intersection lattice} of
the arrangement.
However, as 
Example~\ref{example_different_dimensions}
shows there are arrangements which do not have
an intersection lattice.
}
\end{example}

Finally, we define a {\em quasi-graded intersection poset}
$(P,\rho,\zetabar)$
of a manifold arrangement $\{N_{i}\}_{i=1}^{m}$ of a manifold $M$
to consist of
(i) an intersection poset $P$ of the arrangement,
(ii) the rank function~$\rho$ given by
$\rho(x) = \dim(M) - \dim(x)$
and
$\rho(\emptyset) = \dim(M) + 1$,
and
(iii) the weighted zeta function $\zetabar$
given by the classical zeta function $\zeta$.

\section{The complement of a manifold arrangement}
\label{section_complement}

We define the {\em manifold Zaslavsky invariant}
of a quasi-graded poset $(P,\rho,\zetabar)$
with respect to a function~$f$ defined on~$P$
to be
$$   Z_{M}(P,\rho,\zetabar; f)
     =
       \sum_{\hz \leq x \leq \ho}
                  (-1)^{\rho(x)}
               \cdot
                   \mubar(\hz,x)
               \cdot
                   \zetabar(x,\ho)
               \cdot
                   f(x)   .     $$
In our applications the elements of the poset
will be geometric objects and
the function $f$ will be the Euler characteristic $\chi$.
In the case when $f$ is integer-valued
the manifold Zaslavsky invariant~$Z_{M}$ is an integer.

\begin{theorem}
Let $\{N_{i}\}_{i=1}^{m}$ be a manifold arrangement
in a manifold $M$ with
quasi-graded intersection poset~$(P,\rho,\zeta)$,
and where $\chi$ is the Euler characteristic
of the elements in the intersection poset.
Then the Euler characteristic of the complement
is given by
$$  \chi\left(
          M - \bigcup_{i=1}^{m} N_{i}
            \right)
 =
Z_{M}(P, \rho, \zeta; \chi)  .
$$
\label{theorem_complement}
\end{theorem}
\begin{proof}
For a manifold $x$ in the intersection poset $P$, define
$x^{\circ}$ by
\begin{equation}
   x^{\circ} = x - \bigcup_{y \subseteq x} y  = x - \bigcup_{y \geq x} y  , 
\label{equation_interior}
\end{equation}
that is, 
$x^{\circ}$ consists of all points in $x$ not contained
in any submanifold in $P$.
Observe that $x^{\circ}$ is a manifold
that is not necessarily compact.
Directly for all submanifolds $x$ in the 
intersection poset we have the following disjoint union:
$$    x   =   \bigcup_{x \leq y}^{\bullet} y^{\circ}  .  $$
Applying the Euler characteristic with compact support,
and using the fact $\chi_{c}$ is
additive on disjoint unions, we have
$$    \chi_{c}(x)   =   \sum_{x \leq y} \chi_{c}(y^{\circ})  .  $$
M\"obius inversion yields
$$    
     \chi_{c}(y^{\circ})   =   \sum_{y \leq x} \mu(y,x) \cdot \chi_{c}(x)  .  
$$
By setting $y$ to be the entire manifold $M$,
that is, the minimal element in the intersection
poset~$P$, using
Proposition~\ref{proposition_sign}
and observing that
$(-1)^{\dim(M)} = (-1)^{\rho(x)} \cdot (-1)^{\dim(x)}$,
the result follows.
\end{proof}

Theorem~\ref{theorem_complement}
is an extension of Zaslavsky's classical result on enumerating
the number of regions of a hyperplane arrangement~\cite{Zaslavsky_1}.
Before stating his result, 
we define the {\em Zaslavsky invariant}
of a quasi-graded poset $(P,\rho,\zetabar)$
to be
$$   Z(P,\rho,\zetabar)
     =
       \sum_{\hz \leq x \leq \ho}
                  (-1)^{\rho(x)}
               \cdot
                   \mubar(\hz,x)
               \cdot
                   \zetabar(x,\ho)   .  $$
See~\cite{Ehrenborg_Readdy_Slone} for the graded poset case.
Zaslavsky's immortal result can now be
stated as follows.
\begin{theorem}
Let $\{V_{i}\}_{i=1}^{m}$ be a hyperplane arrangement
in $\Rrr^{n}$ with intersection lattice $L$.
Then the number of chambers in the complement
of the hyperplane arrangement
is given by $Z(L,\rho,\zeta)$.
\label{theorem_Zaslavsky}
\end{theorem}

For our purposes we need to 
extend Theorem~\ref{theorem_Zaslavsky}
to subspace arrangements:
\begin{theorem}
Let $\{V_{i}\}_{i=1}^{m}$ be a subspace arrangement in $\Rrr^{n}$
with quasi-graded intersection poset $(P,\rho,\zeta)$
and let $S^{n-1}$ be an $(n-1)$-dimensional sphere centered at the origin.
Then the Euler characteristic of
the complement of the arrangement in the sphere $S^{n-1}$ is given by
$$
\chi\left( S^{n-1} - \bigcup_{i=1}^{m} V_{i} \right)
     =
Z(P, \rho,\zeta) . $$
\label{theorem_sphere}
\end{theorem}
\begin{proof}
Observe that $\{S^{n-1} \cap V_{i}\}_{i=1}^{m}$ is
a spherical arrangement on the sphere $S^{n-1}$.
Furthermore, the subspace arrangement and the
spherical arrangement have the same
intersection poset $P$.
Hence by Theorem~\ref{theorem_complement}
the Euler characteristic of the complement is given by
\begin{align*}
\chi\left( S^{n-1} - \bigcup_{i=1}^{m} V_{i} \right)
& = 
       \sum_{x \in P}
            (-1)^{\rho(x)}
                 \cdot
            \mu(\hz,x)
                 \cdot
            \chi(x)   \\
& = 
       \sum_{x \in P}
            (-1)^{\rho(x)}
                 \cdot
            \mu(\hz,x)
                 \cdot
            \left((-1)^{\dim(x)} + 1\right)  \\
& = 
       \sum_{x \in P}
       \left(
            (-1)^{n}
          +
            (-1)^{\rho(x)}
        \right)
                 \cdot
            \mu(\hz,x) \\
& = 
       \sum_{x \in P}
            (-1)^{\rho(x)}
                 \cdot
            \mu(\hz,x)  ,
\end{align*}
where in the third step we use
$\rho(x) + \dim(x) = n$ and
in the fourth step we use
$\sum_{x \in P} \mu(\hz,x) = 0$.
\end{proof}

We should be mindful that
$Z(P,\rho,\zetabar)$ only depends on the quasi-graded poset structure,
whereas $Z_{M}(P,\rho,\zetabar;f)$ also depends on the
function values $f(x)$ for elements $x$ in $P$.

\section{The $\cd$-index and Whitney stratifications}
\label{section_cd}

In this section we review the theory of the $\cd$-index
for Eulerian quasi-graded posets
and the important subclass of 
face posets of manifolds which have Whitney stratified
boundaries.
For more details, see the
article~\cite{Ehrenborg_Goresky_Readdy}.

Let $\av$ and $\bv$ be two non-commutative variables.
Given a quasi-graded poset $(P,\rho,\zetabar)$,
the {\em weight} of a chain
$c = \{\hz = x_{0} < x_{1} < \cdots < x_{k} = \ho\}$
is
\begin{equation}
\wt(c)
=
(\av-\bv)^{\rho(x_{0},x_{1})-1}
\cdot
\bv     
\cdot
(\av-\bv)^{\rho(x_{1},x_{2})-1}
\cdot
\bv     
\cdots
\bv     
\cdot
(\av-\bv)^{\rho(x_{k-1},x_{k})-1}  ,
\label{equation_weight}
\end{equation}
where the $\rho(x,y)$ denotes the difference $\rho(y) - \rho(x)$.
Note that the weight of a chain is an $\ab$-polynomial
with integer coefficients homogeneous
of degree $\rho(\hz,\ho) - 1 = \rho(P) - 1$.
Furthermore,
the {\em weighted zeta function} of the chain $c$ is the product
$$   \zetabar(c)
=
\zetabar(x_{0},x_{1})
\cdot
\zetabar(x_{1},x_{2})
\cdots
\zetabar(x_{k-1},x_{k})  .  $$
The {\em $\ab$-index} of the quasi-graded poset $(P,\rho,\zetabar)$
is defined to be
$$   \Psi(P,\rho,\zetabar)
    =
        \sum_{c} \zetabar(c) \cdot \wt(c)   ,  $$
where the sum ranges over all chains $c$
in the quasi-graded poset $P$.
Similarly, the $\ab$-index of
a quasi-graded poset $(P,\rho,\zetabar)$
is an $\ab$-polynomial 
homogeneous of degree $\rho(P) - 1$
with integer coefficients.

\begin{remark}
{\rm
An alternative definition of the $\ab$-index of a quasi-graded
poset of rank $n+1$ is to define
the flag $\fbar$-vector by
$\fbar(S)
=
\sum_{c}
\zetabar(x_{0},x_{1})
\cdot
\zetabar(x_{1},x_{2})
\cdots
\zetabar(x_{k-1},x_{k})$,
where the sum is over all chains
$c = \{\hz = x_{0} < x_{1} < \cdots < x_{k} = \ho\}$
satisfying $S = \{\rho(x_{1}), \ldots, \rho(x_{k-1})\}$.
The flag $\hhbar$-vector is then given by
$\hhbar(S) = \sum_{T \subseteq S} (-1)^{|S-T|} \cdot \fbar(T)$.
Finally, the $\ab$-index is the sum
$\Psi(P,\rho,\zetabar)
=
\sum_{S} \hhbar(S) \cdot u_{S}$, where
the monomial $u_{S} = u_{1} u_{2} \cdots u_{n}$
is given by $u_{i} = \av$ if $i \not\in S$ and
$u_{i} = \bv$ if $i \in S$.
}
\end{remark}

In~\cite[Section~3]{Ehrenborg_Goresky_Readdy}
the definition of Eulerian poset is extended
to quasi-graded posets.
A quasi-graded poset $(P,\rho,\zetabar)$
is {\em Eulerian} if for all elements $x < z$ in the poset $P$
the following equality holds:
\begin{equation}
  \sum_{x \leq y \leq z}
            (-1)^{\rho(x,y)} \cdot \zetabar(x,y) \cdot \zetabar(y,z)
      = 0 .  
\label{equation_Eulerian}
\end{equation}
From~\cite[Theorem~4.2]{Ehrenborg_Goresky_Readdy}
we have that
\begin{theorem}[Ehrenborg--Goresky--Readdy]
The $\ab$-index of an Eulerian quasi-graded poset $(P,\rho,\zetabar)$
can be written in terms of $\cv = \av+\bv$ and
$\dv = \av \cdot \bv + \bv \cdot \av$.
\label{theorem_cd}
\end{theorem}
When the $\ab$-index is expressed in terms of
$\cv$ and $\dv$, we call it the {\em $\cd$-index}.
Also note that the variable $\cv$ has degree $1$ and 
$\dv$ has degree $2$. Like the $\ab$-index,
the $\cd$-index satisfies $\Psi(P,\rho,\zetabar) \in \zcd$.

Observe that when $P$ is a graded poset
and the weighted zeta function $\zetabar$
is the classical zeta function~$\zeta$,
condition~\eqref{equation_Eulerian}
reduces to the classical notion of
an Eulerian poset.
See~\cite[Section~3]{Ehrenborg_Goresky_Readdy}
for a detailed discussion.
Furthermore,
Theorem~\ref{theorem_cd} reduces to the
usual notion of the
$\cd$-index~\cite{Bayer_Klapper,Stanley}.

Define the involution $u \longmapsto u^{*}$
on $\ab$-polynomials
by reading each monomial in reverse, that is,
$(u_{1} u_{2} \cdots u_{n})^{*} = u_{n} \cdots u_{2} u_{1}$
where each $u_{i}$ is either $\av$ or $\bv$.
Observe that this involution restricts to $\cd$-polynomials
as well, since $\cv^{*} = \cv$
and $\dv^{*} = \dv$. Define the dual of
quasi-graded poset $(P,\rho,\zetabar)$ to be
$(P^{*},\rho^{*},\zetabar^{*})$ where $P^{*}$ is the dual
poset, that is, $x \leq_{P^{*}} y$ if $y \leq_{P} x$,
$\rho^{*}(x) = \rho(x,\ho)$ and 
$\zetabar^{*}(x,y) = \zetabar(y,x)$.
Directly it follows
$$   \Psi(P^{*},\rho^{*},\zetabar^{*})
 =
\Psi(P,\rho,\zetabar)^{*} . $$

We now review the notion of a
Whitney stratification.
\begin{definition}
A stratification $\Omega$ of a manifold $M$ is 
a disjoint union of smaller manifolds, called strata,
whose union is $M$.
We assume that the strata 
satisfy the {\em condition of the frontier,} that is,
for two strata $X$ and $Y$ in $\Omega$, we have
\begin{equation}
  X \cap \overline{Y} \neq \emptyset
\Longleftrightarrow
  X  \subseteq \overline{Y}   .
\label{equation_condition_of_the_frontier}
\end{equation}
This condition defines a partial order
on the set of strata, that is, it defines the face poset~$P$
of the stratification by
$X  \leq_{P} Y$ if $X  \subseteq \overline{Y}$.
Furthermore, for the stratification $\Omega$
to be a {\em Whitney stratification}, 
each strata
has to be a (locally closed, not necessarily connected)
smooth submanifold of $M$
and $\Omega$ must satisfy
Whitney's conditions $(A)$ and $(B)$:
\begin{quote}
Let $X <_{\mathcal P} Y$
and suppose $y_{i} \in Y$ is a sequence of points converging to some
$x \in X$ and that $x_{i} \in X$ converges to $x$.
Also assume that (with respect to some local
coordinate system on the manifold $M$) the secant lines
$\ell_{i} = \overline{x_{i} y_{i}}$ converge to some
limiting line~$\ell$ and the tangent planes
$T_{y_{i}} Y$ converge to some limiting plane
$\tau$.
Then the following inclusions hold:
\end{quote}
\begin{equation}
    \text{ (A) } \: T_{x} X \subseteq \tau
\:\:\:\:\:\:\:\:
    \text{  and  }
\:\:\:\:\:\:\:\:
    \text{ (B) } \: \ell \subseteq \tau.
\end{equation}
\end{definition}
We refer the reader
to~\cite{du_Plessis_Wall,Gibson_Wirthmuller_du_Plessis_Looijenga,
Goresky_MacPherson,Mather} for a more detailed discussion.
Note that we allow our strata to be disconnected.
This will be essential in Section~\ref{section_induced}.

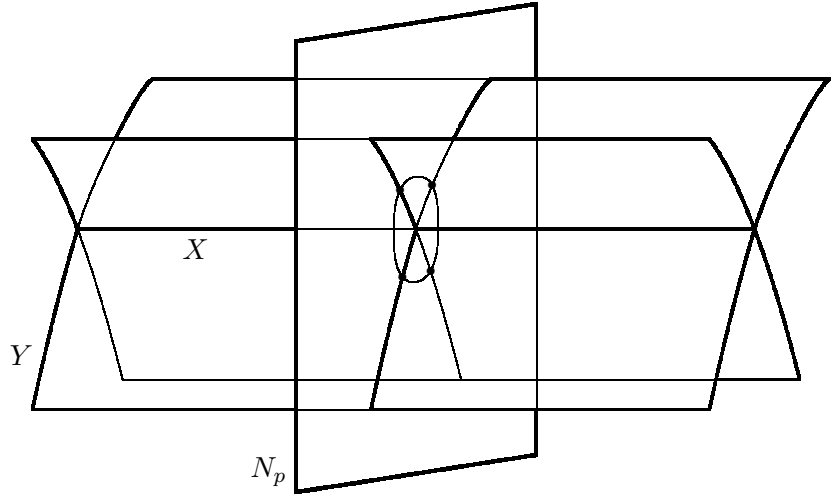
\begin{figure}
\setlength{\unitlength}{1 mm}
\begin{center}
\begin{picture}(106,65)(0,5)

\linethickness{0.30mm}\qbezier(0,16)(3,30)(6,40)
\linethickness{0.10mm}\qbezier(6,40)(7.5,45)(11,52)
\linethickness{0.30mm}\qbezier(11,52)(14.5,59)(16,60)
\linethickness{0.10mm}\qbezier(12,20)(9,32)(6,40)
\linethickness{0.30mm}\qbezier(6,40)(3,48)(0,52)

\linethickness{0.30mm}\qbezier(45,16)(48,30)(51,40)
\linethickness{0.10mm}\qbezier(51,40)(52.5,45)(56,52)
\linethickness{0.30mm}\qbezier(56,52)(59.5,59)(61,60)
\linethickness{0.10mm}\qbezier(57,20)(54,32)(51,40)
\linethickness{0.30mm}\qbezier(51,40)(48,48)(45,52)

\linethickness{0.30mm}\qbezier(90,16)(93,30)(96,40)
\linethickness{0.30mm}\qbezier(96,40)(97.5,45)(101,52)
\linethickness{0.30mm}\qbezier(101,52)(104.5,59)(106,60)
\linethickness{0.30mm}\qbezier(102,20)(99,32)(96,40)
\linethickness{0.30mm}\qbezier(96,40)(93,48)(90,52)

\linethickness{0.40mm}
\qbezier(35,5)(35,35)(35,65)
\qbezier(35,5)(35,5)(67,10)
\qbezier(35,65)(35,65)(67,70)
\qbezier(67,10)(67,10)(67,16)
\qbezier(67,60)(67,60)(67,70)
\linethickness{0.10mm}
\qbezier(67,16)(67,38)(67,60)

\linethickness{0.30mm}
\qbezier(0,16)(17.5,16)(35,16)
\qbezier(6,40)(20.5,40)(35,40)
\qbezier(0,52)(17.5,52)(35,52)
\qbezier(16,60)(25.5,60)(35,60)

\linethickness{0.10mm}
\qbezier(12,20)(23.5,20)(35,20)

\linethickness{0.10mm}
\qbezier(35,16)(40,16)(45,16)
\qbezier(35,40)(43,40)(51,40)
\qbezier(35,52)(91.5,52)(56,52)
\qbezier(35,60)(48,60)(61,60)
\qbezier(35,20)(46,20)(57,20)

\linethickness{0.30mm}
\qbezier(45,16)(67.5,16)(90,16)
\qbezier(51,40)(73.5,40)(96,40)
\qbezier(61,60)(83.5,60)(106,60)
\qbezier(45,52)(67.5,52)(90,52)

\qbezier(91,20)(95.5,20)(102,20)

\linethickness{0.10mm}
\qbezier(57,20)(57,20)(91,20)

\put(41,50)
{\begin{picture}(20,20)(-10,10)
\linethickness{0.10mm}
\qbezier(3,0.46875)(3,7.46875)(0,7)
\qbezier(-3,-0.46875)(-3,6.53125)(0,7)
\qbezier(-3,-0.46875)(-3,-7.46875)(0,-7)
\qbezier(3,0.46875)(3,-6.53125)(0,-7)

\newcommand{\linkdot}{\circle*{1.0}}
\put(2.15,5.85)\linkdot
\put(-2.1,5.2)\linkdot
\put(-1.78,-6.3)\linkdot
\put(1.95,-5.6)\linkdot
\end{picture}}

\put(20,36){$X$}
\put(29,7){$N_{p}$}
\put(-3,22){$Y$}

\end{picture}
\end{center}
\caption{The link of the horizontal line $X$
in the two-dimensional strata $Y$
consists of $4$ points.}
\label{figure_two}
\end{figure}

One important consequence of a Whitney stratification
is that the link of a strata in another strata is well-defined.
Let $X$ be a $k$-dimensional strata and $p$ a point in $X$.
Let $Y$ be another strata such that $X \leq Y$.
Let $N_{p}$ be a normal slice at $p$ to $X$,
that is, a submanifold
such that
$\dim(X) + \dim(N_{p}) = \dim(M)$
and
$X \cap N_{p} = \{p\}$.
Let $B_{\epsilon}(p)$ be a small ball centered at $x$
of radius~$\epsilon > 0$.
Then the homeomorphism type of the intersection
\begin{equation}
     Y \cap N_{p} \cap \partial B_{\epsilon}(p)   
\label{equation_link}
\end{equation}
does not depend on the choice of the point $p$ in $X$,
the choice of the normal slice $N_{p}$ or the choice
of the radius of the ball $B_{\epsilon}(p)$
for small enough $\epsilon > 0$.
The above intersection~\eqref{equation_link} is defined to be the
{\em link of $X$ in $Y$}, denoted by $\link_{Y}(X)$.
For details, see~\cite[Section~6]{Ehrenborg_Goresky_Readdy}.
As an example see Figure~\ref{figure_two},
where $X$ is a one-dimensional strata
and $Y$ is the two-dimensional strata
consisting of $4$~sheets attached to $X$.

Ehrenborg, Goresky and Readdy provided
a geometric source of Eulerian quasi-graded posets,
namely, those arising from Whitney stratified manifolds.
See~\cite[Theorem~6.10]{Ehrenborg_Goresky_Readdy}.
\begin{theorem}[Ehrenborg--Goresky--Readdy]
Let $M$ be a manifold whose boundary has a Whitney
stratification. Let $T$ be the face poset of this stratification
where the partial order relation is given
by~\eqref{equation_condition_of_the_frontier}.
Let the rank function $\rho$ and the weighted zeta function $\zetabar$ be
$$   \rho(x)
=
\begin{cases}
\dim(x) + 1 & \text{ if } x > \hz, \\
0                & \text{ if } x = \hz,
\end{cases}
\:\:\:\: \:\:\:\:
   \zetabar(x,y)
=
\begin{cases}
\chi(\link_{y}(x))
& \text{ if } x > \hz, \\
\chi(y)
& \text{ if } x = \hz.
\end{cases}
$$
Then the quasi-graded face poset
$(T,\rho,\zetabar)$ is Eulerian.
\label{theorem_Whitney_Eulerian}
\end{theorem}

We end this section with a result
about stratifications.
\begin{proposition}
Let $M$ be a manifold with a Whitney stratification $\Omega$
in its boundary.
Assume that there are two strata $X$ and $Y$ of the same
dimension satisfying:
\begin{itemize}
\item[(i)]
for all strata $V$ in $\Omega$
the condition $X < V$ is equivalent to $Y < V$, and
\item[(ii)]
for all strata $V$ in $\Omega$
such that $X < V$,
the two links $\link_{V}(X)$ and $\link_{V}(Y)$ are homeomorphic.
\end{itemize}
Then
$$ \Omega^{\prime} = \Omega - \{X,Y\} \cup \{X \cup Y\} $$
is also a Whitney stratification and
their $\cd$-indexes are equal:
$$  \Psi(\Omega) = \Psi(\Omega^{\prime})  . $$
\label{proposition_X_Y}
\end{proposition}
This result follows from
Lemma~5.4 in~\cite{Ehrenborg_Goresky_Readdy},
which shows that we can replace two elements in a quasi-graded
poset with their union if their up-sets
and their weighted zeta functions agree.

\section{Coalgebraic techniques and the operator ${\cal G}$}
\label{section_coalgebra}

Following~\cite{Ehrenborg_Readdy},
on the algebra of non-commutative polynomials
in the variables $\av$ and $\bv$ we define the coproduct
$\Delta : \zab \longrightarrow \zab \tensor \zab$
by
$$
\Delta(u_{1} u_{2} \cdots u_{k})
=
\sum_{i=1}^{k}
u_{1} \cdots u_{i-1} \tensor u_{i+1} \cdots u_{k}  ,
$$
where $u_{1} u_{2} \cdots u_{k}$ is an $\ab$-monomial
of length $k$
and extend $\Delta$ by linearity.
Note that $\Delta(1) = 0$.
Observe that this coproduct satisfies
the {\em Newtonian condition}:
\begin{equation}
\Delta(u \cdot v)
     =
\sum_{u} u_{(1)} \tensor u_{(2)} \cdot v
      +
\sum_{v} u \cdot v_{(1)} \tensor v_{(2)}   . 
\label{equation_Newtonian}
\end{equation}

Next we have the following result;
see~\cite[Theorem~2.5]{Ehrenborg_Goresky_Readdy}.
\begin{theorem}[Ehrenborg--Goresky--Readdy]
For a quasi-graded poset $(P,\rho,\zetabar)$,
$$  \Delta(\Psi(P,\rho,\zetabar))
    =
      \sum_{\hz < x < \ho}
          \Psi([\hz,x],\rho,\zetabar)
               \tensor
          \Psi([x,\ho],\rho_{x},\zetabar)   ,  $$
where the rank function $\rho_{x}$ is given
by $\rho_{x}(y) = \rho(y) - \rho(x)$.
\label{theorem_coalgebra_map}
\end{theorem}
This result can be stated as
the $\ab$-index is a coalgebra map.
Namely, let $C$ be the $\Zzz$-module generated
by all isomorphism types of quasi-graded posets
and extend the $\ab$-index to be a linear map
$\Psi : C \longrightarrow \zab$.
Let $C$ be a coalgebra by defining the coproduct by
$\Delta(P,\rho,\zetabar)
    =
      \sum_{\hz < x < \ho}
            ([\hz,x],\rho,\zetabar)
                 \tensor
            ([x,\ho],\rho_{x},\zetabar)$.
Theorem~\ref{theorem_coalgebra_map} now states
$\Delta \circ \Psi = (\Psi \tensor \Psi) \circ \Delta$,
that is, $\Psi$ is a coalgebra homomorphism.

We now introduce a number of operators on $\zab$
that will be essential to describe the $\cd$-index of
manifold arrangements.
The operators, $\kappa$, $\eta$, $\varphi$ and $\omega$
were first introduced in~\cite{Billera_Ehrenborg_Readdy}
when studying flag vectors of oriented matroids.

Define the two algebra maps $\kappa$ and $\overline{\lambda}$
such that $\kappa(1) = \overline{\lambda}(1) = 1$ and
$$   \kappa(\av) = \av-\bv, \:\:\:\:
       \kappa(\bv) = 0, \:\:\:\:
       \overline{\lambda}(\av) = 0, \:\:\:\: \text{ and } \:\:\:\:
       \overline{\lambda}(\bv) = \av-\bv . $$
We use the notation $\overline{\lambda}$
to be consistent with the notation $\lambda$
in~\cite{Ehrenborg_Readdy_balanced}.

Define $\eta : \zab \longrightarrow \zab$
by
$$    \eta(w)
     =
        \begin{cases}
                2 \cdot (\av-\bv)^{m+k}
                     & \text{ if } w = \bv^{m} \cdot \av^{k} , \\
                0
                     & \text{ otherwise.}
        \end{cases}
$$
\begin{lemma}
Let $(P,\rho,\zeta)$ be a quasi-graded poset
where the weighted zeta function is the classical zeta function $\zeta$.
Then the following identities hold:
\begin{align}
\kappa(\Psi(P,\rho,\zeta))
& = 
(\av-\bv)^{\rho(P)-1}  , 
\label{equation_kappa} \\
\overline{\lambda}(\Psi(P,\rho,\zeta))
& = 
(-1)^{\rho(P)} \cdot \mu(P) \cdot (\av-\bv)^{\rho(P)-1}  , 
\label{equation_lambda} \\
\eta(\Psi(P,\rho,\zeta))
& = 
Z(P,\rho,\zeta) \cdot (\av-\bv)^{\rho(P)-1}  .
\label{equation_eta}
\end{align}
\end{lemma}
\begin{proof}
Equation~\eqref{equation_kappa}
is a direct observation.
See also~\cite[Equation~(5)]{Billera_Ehrenborg_Readdy}.
Equation~\eqref{equation_lambda}
is a reformulation of Hall's formula for the
M\"obius function.
Finally~\eqref{equation_eta}
follows
from~\cite[Lemma~5.2 and Equation~(6)]{Billera_Ehrenborg_Readdy}.
Although this reference only proves this relation for classical
graded posets, the proof for quasi-graded posets carries
through using the same techniques.
\end{proof}

Define the operator $\varphi$ as the sum $\sum_{k \geq 1} \varphi_{k}$,
where $\varphi_{k}$ is defined by the $k$-ary coproduct
$$    \varphi_{k}(w)
     =
        \sum_{w}
           \kappa(w_{(1)})
           \cdot \bv \cdot
           \eta(w_{(2)})
           \cdot \bv \cdots \bv \cdot
           \eta(w_{(k)})    .      $$
From~\cite[Lemmas~5.6 and~5.7]{Billera_Ehrenborg_Readdy}
we have the following lemma for evaluating $\varphi$.
\begin{lemma}
Let $v$ be an $\ab$-polynomial without a constant term.
Then $\varphi(v \cdot \mab) = \varphi(v) \cdot 2\dv$.
Furthermore,
let $x$ be either $\av$ or $\bv$
and assume that
the monomial $v \cdot x$ does not end with $\mab$.
Then $\varphi(v \cdot x) = \varphi(v) \cdot \cv$.
\label{lemma_varphi}
\end{lemma}

Define the linear map $\omega : \zab \longrightarrow \zcd$
by first defining $\omega$
on a monomial by replacing each occurrence of
$\av\bv$ by $2\dv$ and then replacing the remaining letters
by $\cv$. Extend by linearity to all $\ab$-polynomials in $\zab$.
As an example, $\omega(\av\bv\bv\av) = 2 \dv \cv^{2}$.
The $\omega$ map is equivalent to
Stembridge's peak algebra map $\theta$~\cite{Stembridge}.

From~\cite[Proposition~5.5]{Billera_Ehrenborg_Readdy}
we have:
\begin{proposition}
The two linear operators $\varphi$ and $\omega$ agree
on $\ab$-monomials that begin with the letter $\av$, that is,
$\varphi(\av \cdot v) = \omega(\av \cdot v)$.
\end{proposition}

Now define the operator ${\cal G} : \zab \longrightarrow \zab$
by the relation
$$  {\cal G}(w)
    =
    \varphi(w) \cdot \bv
+
\sum_{w}
    \varphi(w_{(1)})
              \cdot \bv \cdot       
\overline{\lambda}(w_{(2)})
\cdot (\av-\bv)  .  $$
Applying this operator to the $\ab$-index
of a quasi-graded poset $P$
whose weighted zeta function is
the zeta function, we have
\begin{equation}
{\cal G}\big(\Psi(P)\big)
=
\varphi\big(\Psi(P)\big) \cdot \bv
+
\sum_{\hz < y < \ho}
    \varphi\big(\Psi([\hz,y])\big) 
              \cdot \bv \cdot       
\overline{\lambda}\big(\Psi([y,\ho])\big)
\cdot (\av-\bv) .
\label{equation_G_Psi}
\end{equation}

\begin{lemma}
For an $\ab$-polynomial $w$ we have
$$    {\cal G}(w \cdot \av) =     \frac{1}{2} \cdot \varphi(w \cdot \av \bv) . $$
\label{lemma_ending_a}
\end{lemma}
\begin{proof}
By the Newtonian condition~\eqref{equation_Newtonian} we have that
$\Delta(w \cdot \av)
=
w \tensor 1
+
\sum_{w} w_{(1)} \tensor w_{(2)} \cdot \av$.
Hence we have
\begin{align*}
{\cal G}(w \cdot \av)
  & = 
    \varphi(w \cdot \av) \cdot \bv
+
    \varphi(w)
              \cdot \bv \cdot       
\overline{\lambda}(1)
\cdot (\av-\bv) 
+
\sum_{w}
    \varphi(w_{(1)})
              \cdot \bv \cdot       
\overline{\lambda}(w_{(2)} \cdot \av)
\cdot (\av-\bv)  \\
  & = 
    \varphi(w) \cdot \cv \cdot \bv
+
    \varphi(w)
              \cdot \bv \cdot (\av-\bv)  \\
  & = 
    \varphi(w) \cdot \dv \\
  & = 
    1/2 \cdot \varphi(w \cdot \av \bv) ,
\end{align*}
where in the
second step we used
Lemma~\ref{lemma_varphi},
$\overline{\lambda}(1) = 1$
and
$\overline{\lambda}(w_{(2)} \cdot \av)
=
\overline{\lambda}(w_{(2)})
\cdot
\overline{\lambda}(\av)
=
0$,
and
in the fourth step Lemma~\ref{lemma_varphi} was applied again.
\end{proof}

\begin{proposition}
For any $\ab$-polynomial $v$ the operator ${\cal G}$ satisfies
\begin{align}
{\cal G}(1) & =  \bv ,
\label{equation_G_1} \\
{\cal G}(\av \cdot v) & =  \frac{1}{2} \cdot \omega(\av \cdot v \cdot \bv)  .
\label{equation_G_a}
\end{align}
\label{proposition_G}
\end{proposition}
\begin{proof}
It is a straightforward
verification that ${\cal G}(1) = \bv$.
Next we prove statement~\eqref{equation_G_a}
by induction.
The induction basis is $v = 1$, which
follows from
${\cal G}(\av)
=
\varphi(\av) \cdot \bv 
+
\varphi(1) \cdot \bv \cdot \overline{\lambda}(1) \cdot (\av-\bv)
=
\cv \cdot \bv
+
\bv \cdot (\av-\bv)
=
\dv
=
1/2 \cdot \varphi(\av\bv)$.
Assume now that the statement holds for $v$
and let $w = \av \cdot v$.
By Lemma~\ref{lemma_ending_a}
we know it is true for $v \cdot \av$.
The last case to consider is~$v \cdot \bv$
and again use $w = \av \cdot v$.
The Newtonian condition~\eqref{equation_Newtonian}
implies
$\Delta(w \cdot \bv)
=
w \tensor 1
+
\sum_{w} w_{(1)} \tensor w_{(2)} \cdot \bv$.
Now
\begin{align*}
{\cal G}(w \cdot \bv)
  & = 
    \varphi(w \cdot \bv) \cdot \bv
+
    \varphi(w)
              \cdot \bv \cdot       
\overline{\lambda}(1)
\cdot (\av-\bv) 
+
\sum_{w}
    \varphi(w_{(1)})
              \cdot \bv \cdot       
\overline{\lambda}(w_{(2)} \cdot \bv)
\cdot (\av-\bv)  \\
  & = 
    \varphi(w \cdot \bv) \cdot \bv
+
\Big(
    \varphi(w)
              \cdot \bv
+
\sum_{w}
    \varphi(w_{(1)})
              \cdot \bv \cdot       
\overline{\lambda}(w_{(2)})
\cdot (\av-\bv)
\Big)
\cdot (\av-\bv)  \\
  & = 
    \varphi(w \cdot \bv) \cdot \bv
+
{\cal G}(w)
\cdot (\av-\bv)  \\
  & = 
    \varphi(w \cdot \bv) \cdot \bv
+
1/2 \cdot \varphi(w \cdot \bv)  \cdot (\av-\bv) \\
  & = 
1/2 \cdot \varphi(w \cdot \bv)  \cdot \cv \\
  & = 
1/2 \cdot \varphi(w \cdot \bv \bv) ,
\end{align*}
where the fourth step is the induction hypothesis
and the sixth step is Lemma~\ref{lemma_varphi},
completing the induction.
\end{proof}

\section{The induced stratification}
\label{section_induced}

Let $M$ be an $n$-dimensional manifold and
let $\{N_{i}\}_{i=1}^{m}$ be a manifold
arrangement in the boundary of~$M$.
The arrangement induces a Whitney
stratification of the boundary of $M$ as follows.
Recall that for an intersection
$x = \bigcap_{i \in I} N_{i}$
we let 
$x^{\circ}$ be all points in $x$
not contained in any submanifold of~$x$;
see~\eqref{equation_interior}.
Now the induced subdivision $T$ is the collection of all
connected components of 
$\left(\bigcap_{i \in I} N_{i}\right)^{\circ}$,
where $I$ ranges over all index sets,
together with
the empty strata $\emptyset$, and
the manifold~$M$ as the maximal strata.
Observe that the empty index set yields
the connected components of $(\partial M)^{\circ}$.

\begin{proposition}
The stratification $T$ is a Whitney stratification.
\label{proposition_T_Whitney}
\end{proposition}
\begin{proof}
Pick two strata $X$ and $Y$ from $T$ where $X <_{T} Y$
and a point $x$ in $X$.
Since the clean intersection property holds at the point $x$ we can choose
a local coordinate system around $x$ such that the two strata
$X$ and~$Y$ are locally straight in a neighborhood $U$ around $x$,
that is, for any point $p$ close enough to $x$ the 
tangent planes~$T_{p} X$, respectively  $T_{p} Y$, 
are independent of the point $p$.
Let $y_{i} \in Y$ be any sequence of points
converging to the point $x$. Without loss of
generality, we may assume that $y_{i}$ lies
in the neighborhood $U$.
Since the tangent planes  $T_{y_{i}} Y$ are all the same,
they are in fact equal to the limiting plane $\tau$.
Hence Whitney's condition $(A)$ holds:
$T_{x} X \subseteq T_{x} Y = \tau$.
Similarly, 
let $x_{i} \in X$ be any sequence of points
converging to the point $x$. 
Again, we may assume that $x_{i} \in U$.
Now all the lines~$\ell_{i} = \overline{x_{i} y_{i}}$ lie in the plane $\tau$
so that the limiting line $\ell$ also lies in $\tau$, and thus
Whitney's condition $(B)$ holds.
\end{proof}

\begin{example}
{\rm
Consider the subspace arrangement in $\Rrr^{3}$
consisting of
the two planes $x=0$ and $y=0$, and
the line $x=y=z$.
Intersect this arrangement with $S^{2}$
to obtain a spherical arrangement
consisting of two $1$-dimensional spheres
and one $0$-dimensional sphere.
The face poset of the induced stratification of the sphere
consists of four points
$\pm(0,0,1)$
and
$\pm(1/\sqrt{3},1/\sqrt{3},1/\sqrt{3})$;
four open edges ($1$-dimensional strata)
each emanating from $(0,0,1)$ to $(0,0,-1)$;
and finally, four $2$-dimensional strata
where two of the strata are discs
(the $x$ and $y$ coordinates have different signs)
and the other two strata are punctured discs
(the $x$ and $y$ coordinates have the same sign).
See Figure~\ref{figure_three}
for the face poset $T$.
}
\label{example_one}
\end{example}

\begin{figure}
\setlength{\unitlength}{1mm}

\begin{center}
\begin{picture}(50,40)(0,0)
\qbezier(25,0)(25,0)(0,10)
\qbezier(25,0)(25,0)(20,10)
\qbezier(25,0)(25,0)(30,10)
\qbezier(25,0)(25,0)(50,10)

\qbezier(20,10)(20,10)(10,20)
\qbezier(20,10)(20,10)(20,20)
\qbezier(20,10)(20,10)(30,20)
\qbezier(20,10)(20,10)(40,20)
\qbezier(30,10)(30,10)(10,20)
\qbezier(30,10)(30,10)(20,20)
\qbezier(30,10)(30,10)(30,20)
\qbezier(30,10)(30,10)(40,20)

\qbezier(10,20)(10,20)(10,30)
\qbezier(20,20)(20,20)(10,30)
\qbezier(20,20)(20,20)(30,30)
\qbezier(40,20)(40,20)(30,30)
\qbezier(40,20)(40,20)(40,30)
\qbezier(30,20)(30,20)(40,30)
\qbezier(30,20)(30,20)(20,30)
\qbezier(10,20)(10,20)(20,30)

\qbezier(0,10)(0,10)(10,30)
\qbezier(50,10)(50,10)(40,30)

\qbezier(10,30)(10,30)(25,40)
\qbezier(20,30)(20,30)(25,40)
\qbezier(30,30)(30,30)(25,40)
\qbezier(40,30)(40,30)(25,40)

\put(25,0){\circle*{1}}
\put(0,10){\circle*{1}}
\put(20,10){\circle*{1}}
\put(30,10){\circle*{1}}
\put(50,10){\circle*{1}}
\multiput(10,20)(0,10){2}{\multiput(0,0)(10,0){4}{\circle*{1}}}
\put(25,40){\circle*{1}}

\put(-3,9){$a$}
\put(52,9){$b$}
\put(6,29){$A$}
\put(42,29){$B$}

\end{picture}
\hspace*{25 mm}
\begin{picture}(30,40)(0,0)
\qbezier(5,0)(5,0)(5,10)
\qbezier(5,0)(5,0)(25,10)
\qbezier(5,10)(5,10)(0,20)
\qbezier(5,10)(5,10)(10,20)
\qbezier(0,20)(0,20)(5,30)
\qbezier(10,20)(10,20)(5,30)
\qbezier(25,10)(25,10)(5,30)
\qbezier(5,30)(5,30)(5,40)

\put(5,0){\circle*{1}}
\put(5,10){\circle*{1}}
\put(25,10){\circle*{1}}
\put(0,20){\circle*{1}}
\put(10,20){\circle*{1}}
\put(5,30){\circle*{1}}
\put(5,40){\circle*{1}}

\put(2,-1){$\hz$}
\put(2,9){$c$}
\put(27,9){$d$}
\put(-3,19){$e$}
\put(10.5,19){$f$}
\put(1,29){$C$}
\put(2,39){$\ho$}
\end{picture}
\end{center}
\caption{The face poset $T$ of the induced subdivision
and the quasi-graded poset $Q$
of Example~\ref{example_one}.}
\label{figure_three}
\end{figure}
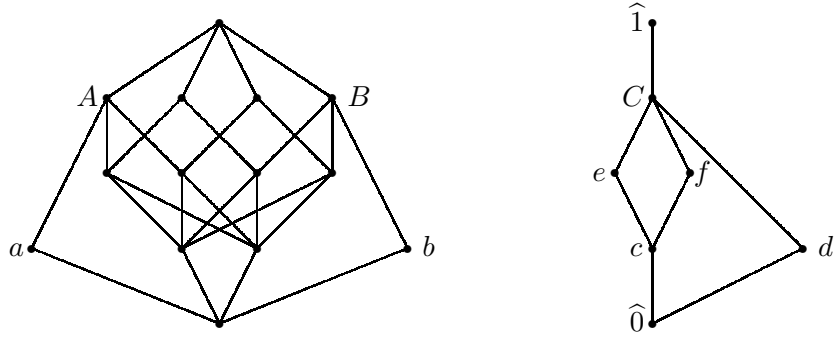

\begin{theorem}
Let $M$ be an $n$-dimensional manifold.
Let $\{N_{i}\}_{i=1}^{m}$ be a manifold
arrangement in the boundary $\partial M$
with an intersection poset $P$.
Let $T$ be the induced Whitney stratification of $M$.
Then the reverse of the $\cd$-index of $T$ is given
\begin{align*}
    \Psi(T)^{*}
  & = 
  \chi(M) \cdot 
\begin{cases}
          (\cv^{2} - 2 \dv)^{n/2}
& \text{if $n$ is even,}  \\
          \cv
             \cdot
          (\cv^{2} - 2 \dv)^{(n-1)/2}
& \text{if $n$ is odd,} 
 \end{cases}   \\
&
 +
      \sum_{\substack{x \in {P}, x > \hz \\ \dim(x) \text{ even}}}
          \frac{1}{2} \cdot 
                   \omega\big(\av \cdot \Psi([\hz,x]) \cdot \bv\big)
             \cdot
          (\cv^{2} - 2 \dv)^{\dim(x)/2}
             \cdot
          \chi(x)    .
\end{align*}
\label{theorem_main}
\end{theorem}

Before proving Theorem~\ref{theorem_main},
we first introduce a quasi-graded poset $Q$
that is different 
than the quasi-graded poset $T$.
This new quasi-graded poset will be smaller than $T$.
However, it will have the same $\cd$-index as $T$.
Let $\PPP = P \cup \{\widehat{-1}\}$ be the intersection poset
with a new minimal element~$\widehat{-1}$ adjoined.

Define $Q$ to be the poset 
$$
Q
=
\{x^{\circ} \:\: : \:\: x \in P\} \cup \{M\}
$$
where the partial order is
$x^{\circ} \leq y^{\circ}$ if $x \subseteq y$.
Observe that this condition is equivalent to
$x^{\circ} \subseteq \overline{y^{\circ}}$ since
$\overline{y^{\circ}} = y$. This is the partial order
of the face poset of the induced stratification
where the strata are of the form~$x^{\circ}$.
Furthermore, note that $M$ is the maximal element
and it covers the element $(\partial M)^{\circ}$,
which is the unique element of rank $n$.

Define the rank function $\rho_{Q}$ by
$$  \rho_{Q}(x^{\circ})
=
\begin{cases}
0 & \text{ if } x^{\circ} = \hz,\\
\dim(x^{\circ}) + 1 & \text{ otherwise,} \\
\end{cases}
$$
and the weighted zeta function $\zetabar_{Q}$ by
$$  \zetabar_{Q}(x^{\circ},y^{\circ})
=
\begin{cases}
\chi(y^{\circ}) & \text{ if } \hz = x^{\circ} , \\
\chi(\link_{y^{\circ}}(x^{\circ})) & \text{ if } \hz < x^{\circ} .
\end{cases}
$$
As a poset $Q$ is the dual poset of $\PPP$
via the map $\psi(x^{\circ}) = x$
and
$\psi(M) = \widehat{-1}$.

There is a natural order-preserving map $z$
from the quasi-graded face poset $T$
to the poset~$Q$.
Namely, for an element $x$ in $T$ let $z(x)$
be the rank-wise smallest element in the poset $Q$
containing~$x$.
Observe the map $z$ also preserves the rank function,
that is, 
$\rho_{T}(x) = \rho_{Q}(z(x))$.
As a side comment, the reason why this map is called $z$ stems 
from oriented matroid theory, as it selects
the coordinates that are equal to zero in
the covectors of the oriented
matroid~\cite[Section~4.6]{OM}.

The same argument as Proposition~\ref{proposition_T_Whitney}
yields:
\begin{proposition}
The stratification $Q$ is a Whitney stratification.
\label{proposition_Q_Whitney}
\end{proposition}
Since the links of a Whitney stratification are well-defined,
we have the next corollary.
\begin{corollary}
Let $x$ and $y$ be two manifolds in the intersection poset
such that $x \subseteq y$.
Furthermore, let $p$ and $q$ be two points in $x^{\circ}$.
Let $N_{p}$ and $N_{q}$
be the normal slices to $x$ at $p$,
respectively at~$q$.
Then the two spaces
\begin{equation}
   y^{\circ} \cap N_{p} \cap \partial B_{\epsilon}(p)
\:\:\:\: \text{ and } \:\:\:\:
   y^{\circ} \cap N_{q} \cap \partial B_{\epsilon}(q)
\label{equation_p_q}
\end{equation}
are homeomorphic for small enough $\epsilon > 0$.
\label{corollary_p_q}
\end{corollary}
\begin{proof}
The two spaces in~\eqref{equation_p_q}
are both homeomorphic to $\link_{y}(x)$.
\end{proof}

\begin{proposition}
The $\cd$-indexes of the two quasi-graded posets
$T$ and $Q$ are equal, that is,
$\Psi(T,\rho_{T},\zetabar_{T}) = \Psi(Q,\rho_{Q},\zetabar_{Q})$.
\label{proposition_T_Q}
\end{proposition}
\begin{proof}
Choose a linear extension of the poset $Q$, that is,
$Q = \{x_{1}, \ldots, x_{k}\}$ such that
$x_{i} \leq_{Q} x_{j}$ implies $i \leq j$.
Note that $x_{0} = \hz$ and $x_{k} = \ho = M$.
Starting with $x_{k-1}$ select two elements
$u$ and~$v$ such that $\psi(z(u)) = \psi(z(v)) = x_{k-1}$.
By Corollary~\ref{corollary_p_q},
the conditions of Proposition~\ref{proposition_X_Y}
are satisfied
and hence 
we can replace $u$ and $v$ by their union $u \cup v$
without changing the $\cd$-index.
Repeat this operation until there are no more
such elements. Continue with elements mapping
to $x_{k-2}$ and work towards the strata $x_{2}$.
Proposition~\ref{proposition_X_Y}
guarantees that at every step the $\cd$-index does not change.
The end result is the stratification~$Q$,
proving the proposition.
\end{proof}

\addtocounter{theorem}{-4}
\begin{continuation}
{\rm
The weighted zeta function of the quasi-graded poset $T$
takes the value $1$ everywhere, except for the
four values 
$\zetabar_{T}(\hz,A) = \zetabar_{T}(\hz,B)
= \zetabar_{T}(a,A) = \zetabar_{T}(b,B) = 0$.

Similarly, the non $1$-values of
the weighted zeta function for the quasi-graded poset $Q$
are given by
$\zetabar_{Q}(\hz,c) = \zetabar_{Q}(\hz,d)
= \zetabar_{Q}(\hz,e) = \zetabar_{Q}(\hz,f)
= \zetabar_{Q}(\hz,C)
= \zetabar_{Q}(c,e) = \zetabar_{Q}(c,f) = 2$,
$\zetabar_{Q}(c,C) = 4$,
$\zetabar_{Q}(d,C) = 0$
and
$\zetabar_{Q}(e,C) = \zetabar_{Q}(f,C) = 2$.
Observe that each of the five strata
$c$, $d$, $e$, $f$ and $C$
are disconnected.

In both cases we obtain
$\Psi(T) = \Psi(Q) = \mccc + 2 \cdot \mdc$.
}
\end{continuation}
\addtocounter{theorem}{3}

We now explicitly describe $\zetabar_{Q}$
in terms of the invariants $Z$ and $Z_{M}$.
\begin{proposition}
The weighted zeta function $\zetabar_{Q}$ is given by
\begin{align*}
\zetabar_{Q}(y^{\circ},\ho)
& = 
1
& & \text{ for }
\widehat{-1} <_{\PPP} y <_{\PPP} \ho, \\
\zetabar_{Q}(y^{\circ},x^{\circ})
& = 
Z([x,y],\rho_{\PPP},\zeta)
& & \text{ for }
\widehat{-1} <_{\PPP} x \leq_{\PPP} y <_{\PPP} \ho, \\
\zetabar_{Q}(\hz,x^{\circ})
& = 
Z_{M}([x,\ho],\rho_{\PPP},\zeta; \chi)
& & \text{ for }
\widehat{-1} <_{\PPP} x <_{\PPP} \ho, \\
\zetabar_{Q}(\hz,\ho)
& = 
\chi(M) .
\end{align*}
\label{proposition_Q}
\end{proposition}
\begin{proof}
The first and fourth equations are direct.
The second equation follows from
\begin{align*}
\zetabar_{Q}(y^{\circ},x^{\circ})
& =
\chi(\link_{x^{\circ}}(y^{\circ}))
=
Z([x,y],\rho_{\PPP},\zeta) ,
\end{align*}
where the second equality is by
Theorem~\ref{theorem_sphere}.
Similarly, we have
\begin{align*}
\zetabar_{Q}(\hz,x^{\circ})
 & =
\chi(x^{\circ})
=
Z_{M}([x,\ho],\rho_{\PPP},\zeta; \chi) .
\qedhere
\end{align*}
\end{proof}

We are now positioned to prove our main result.
\begin{proof}[Proof of Theorem~\ref{theorem_main}.]
Throughout this proof we suppress
the dependency on the rank function~$\rho$
and the zeta function $\zeta$ in our notation.
By summing over all chains $c$
in the poset $\PPP$ we can compute the
$\ab$-index of $Q$ and 
hence of $T$.
However, we prefer to compute the reverse $\ab$-indexes,
that is,
\begin{align}
\Psi(T)^{*}
=
\Psi(Q)^{*}
& = 
\chi(M) \cdot (\av-\bv)^{n+1} 
\nonumber \\
&   
\quad +
\sum_{c}
        Z([x_{1},x_{2}])
           \cdots
        Z([x_{k-2},x_{k-1}])
           \cdot
        Z_{M}([x_{k-1},x_{k}]; \chi)  
\cdot
\wt(c)   ,
\label{equation_T}
\end{align}
where the sum is over all chains
$c = \{\widehat{-1} = x_{0} < x_{1} < \cdots < x_{k} = \ho\}$
of length $k \geq 2$ in $\PPP$.

Recall that
the operator $\eta$ satisfies the relation~\eqref{equation_eta}.
Similarly, 
define~$\eta_{M}$ such that
$$
\eta_{M}\big(\Psi(P)\big) 
  =
Z_{M}(P; \chi) \cdot (\av-\bv)^{\rho(P)-1}  , 
$$
where we suppress
the dependence of $\eta_{M}(\Psi(P))$
on the Euler characteristic of the elements in $P$.
By expanding equation~\eqref{equation_weight}
we can rewrite equation~\eqref{equation_T} as
\begin{align}
\Psi(T)^{*}
& = 
\chi(M) \cdot (\av-\bv)^{n+1}
\nonumber \\
&  +  
\sum_{c}
\kappa\big(\Psi([x_{0},x_{1}])\big)
\cdot \bv \cdot
\eta\big(\Psi([x_{1},x_{2}])\big)
\cdot \bv \cdots \bv \cdot
\eta\big(\Psi([x_{k-2},x_{k-1}])\big)
\cdot \bv \cdot
\eta_{M}\big(\Psi([x_{k-1},x_{k}])\big)
\nonumber \\
\label{equation_kappa_eta_eta_M}
\end{align}
We split the sum by first summing over the element $y=x_{k-1}$
and then over all chains $c^{\prime}$ in the interval $[\widehat{-1},y]$.
By the definition of the operator $\varphi$ 
and the fact that the $\ab$-index is a coalgebra homeomorphism,
we then have
\begin{equation}
\Psi(T)^{*}
=
\chi(M) \cdot (\av-\bv)^{n+1} 
+
\sum_{\widehat{-1} < y < \ho}
\varphi\big(\Psi([\widehat{-1},y])\big)
\cdot \bv \cdot
\eta_{M}\big(\Psi([y,\ho])\big)  .
\label{equation_T2}
\end{equation}
Now apply $\eta_{M}$ to the interval $[y,\ho]$.
We have
\begin{align}
\eta_{M}\big(\Psi([y,\ho])\big) 
 & = 
Z_{M}([y,\ho]; \chi) \cdot (\av-\bv)^{\rho(y,\ho)-1}
\nonumber \\
 & = 
\sum_{y \leq x \leq \ho}
  (-1)^{\rho(y,x)} \cdot \mu(y,x) \cdot (\av-\bv)^{\rho(y,x)-1}
\cdot \chi(x) \cdot (\av-\bv)^{\rho(x,\ho)}
\nonumber \\
 & = 
 \chi(y) \cdot (\av-\bv)^{\rho(y,\ho)-1}
 +
\sum_{y < x < \ho}
\overline{\lambda}\big(\Psi([y,x])\big)
\cdot \chi(x) \cdot (\av-\bv)^{\rho(x,\ho)} ,
\label{equation_eta_M}
\end{align}
where the term $x = \ho$ vanished
since the maximal element $\ho$ is the
empty set with $\chi(\ho) = 0$.
Substituting
equation~\eqref{equation_eta_M}
into equation~\eqref{equation_T2}
we have
\begin{align*}
\Psi(T)^{*}
& = 
\chi(M) \cdot (\av-\bv)^{n+1}
+
\sum_{\widehat{-1} < y < \ho}
\varphi\big(\Psi([\widehat{-1},y])\big)
\cdot \bv \cdot
 \chi(y) \cdot (\av-\bv)^{\rho(y,\ho)-1} \\
 &
\quad +
\sum_{\widehat{-1} < y < \ho}
\sum_{y < x < \ho}
\varphi\big(\Psi([\widehat{-1},y])\big)
\cdot \bv \cdot
\overline{\lambda}\big(\Psi([y,x])\big)
\cdot \chi(x) \cdot (\av-\bv)^{\rho(x,\ho)}  \\
& = 
\chi(M) \cdot (\av-\bv)^{n+1}
+
\sum_{\widehat{-1} < x < \ho}
\varphi\big(\Psi([\widehat{-1},x])\big)
\cdot \bv \cdot
 \chi(x) \cdot (\av-\bv)^{\rho(x,\ho)-1} \\
 &
\quad +
\sum_{\widehat{-1} < x < \ho}
\sum_{\widehat{-1} < y < x}
\varphi\big(\Psi([\widehat{-1},y])\big)
\cdot \bv \cdot
\overline{\lambda}\big(\Psi([y,x])\big)
\cdot \chi(x) \cdot (\av-\bv)^{\rho(x,\ho)}  \\
& = 
\chi(M) \cdot (\av-\bv)^{n+1} \\
&
\quad +
\sum_{\widehat{-1} < x < \ho}
\Big(
    \varphi\big(\Psi([\widehat{-1},x])\big) \cdot \bv
+
\sum_{\widehat{-1} < y < x}
    \varphi\big(\Psi([\widehat{-1},y])\big) 
              \cdot \bv \cdot       
\overline{\lambda}\big(\Psi([y,x])\big)
\cdot (\av-\bv)
\Big) \\
&
\hspace{20 mm}
\cdot \chi(x) \cdot (\av-\bv)^{\rho(x,\ho)-1} \\
  & = 
\chi(M) \cdot (\av-\bv)^{n+1}
+
\sum_{\widehat{-1} < x < \ho}
{\cal G}\big(\Psi([\widehat{-1},x])\big)
\cdot \chi(x) \cdot (\av-\bv)^{\rho(x,\ho)-1} ,
\end{align*}
where in the second step we changed the variable in
the second term and switched the order of summation in the third term.
The last step was to apply the fact
that the $\ab$-index is a coalgebra map and
the definition of the operator
${\cal G}$ to~$\Psi([\widehat{-1},x])$;
see equation~\eqref{equation_G_Psi}.

From the last sum we break out the case when
$x = \hz$, that is, the boundary of $M$.
Recall that ${\cal G}(1) = \bv$.
For $x > \hz$ we use 
that $\Psi([\widehat{-1},x]) = \av \cdot \Psi([\hz,x])$
and we can apply
Proposition~\ref{proposition_G}.
\begin{align*}
\Psi(T)^{*}
& = 
\chi(M) \cdot (\av-\bv)^{n+1}
+
\chi(\partial M) \cdot \bv \cdot (\av-\bv)^{n} \\
&
\quad +
\sum_{\hz < x < \ho}
1/2 \cdot \omega\big(\av \cdot \Psi([\hz,x]) \cdot \bv\big)
\cdot \chi(x) \cdot (\av-\bv)^{\rho(x,\ho)-1} ,
\end{align*}
The result follows now by observing
$\chi(\partial M) = \left(1 - (-1)^{\dim(M)}\right) \cdot \chi(M)$,
the Euler characteristic of 
an odd dimensional manifold $x$ is $0$
(since $x$ has no boundary),
and that $(\av-\bv)^{2} = \cv^{2} - 2\dv$.
\end{proof}

Similar to the $\cv$-$2\dv$-index in~\cite{Billera_Ehrenborg_Readdy}
we have the next result.
\begin{corollary}
Let $M$ be a manifold
and $\{N_{i}\}_{i=1}^{m}$ be a manifold
arrangement in the boundary of $M$, that is, $\partial M$.
Let $T$ be the induced Whitney stratification of $M$.
Let $w$ be a $\cd$-monomial containing $k$ $\dv$'s.
Then the coefficient of $w$ in the $\cd$-index $\Psi(T)$
is divisible by $2^{k-1}$.
\end{corollary}

\section{Spherical and toric arrangements}
\label{section_sphere_torus}

We now turn our attention to consequences of the main result.
As mentioned in the introduction,
we consider two important cases that have
been studied earlier, namely spherical and toric arrangements.
In both cases, the results have been expanded.

\subsection{Spherical arrangements}

We now extend the original results
for spherical arrangements where each sphere
has codimension~$1$ to arrangements without
this restriction.
By a $k$-dimensional sphere we mean a manifold
homeomorphic to the $k$-dimensional unit sphere
$\{(x_{1}, \ldots, x_{k+1}) \in \Rrr^{k+1} 
    \: : \: x_{1}^{2} + \cdots + x_{k+1}^{2} = 1\}$.
Define a {\em spherical arrangement} to be
a collection of spheres 
$\{N_{i}\}_{i=1}^{m}$
on the boundary of a ball
such that the intersection
$\bigcap_{i \in I} N_{i}$ 
is a disjoint union of spheres
for all index sets $I \subseteq \{1, \ldots, m\}$.
An example of a spherical arrangement is to intersect
an arrangement of affine subspaces in
Euclidean space with a sphere
of large enough radius.
Example~\ref{example_different_dimensions},
where two $2$-dimensional spheres
intersect in two points and a circle,
is also a spherical arrangement.

Observe that we require every non-empty element
of an intersection poset to be a sphere.
The next result not only extends the original result
in~\cite[Theorem~3.1]{Billera_Ehrenborg_Readdy},
but also~\cite[Theorem~4.10]{Ehrenborg_Readdy_Slone}.

\begin{theorem}
Let $\{N_{i}\}_{i=1}^{m}$ be a spherical
arrangement in the boundary of an $n$-dimensional ball,
with~$(P,\rho,\zeta)$ as a quasi-graded intersection poset of the
arrangement.
Let $T$ be the induced Whitney stratification.
Then the $\cd$-index of $T$ is given by
\begin{align*}
    \Psi(T)
  & = 
    \omega\big(\av \cdot \Psi(P,\rho,\zeta)\big)^{*}  .
\end{align*}
\label{theorem_spherical}
\end{theorem}
\begin{proof}
Observe that the Zaslavsky invariant and the
manifold Zaslavsky invariant agree for a
spherical arrangements. That is, for an
interval $[y,\ho]$ in $P$ we have 
\begin{align*}
Z_{M}([y,\ho])
& = 
       \sum_{y \leq x \leq \ho}
                  (-1)^{\rho(y,x)}
               \cdot
                   \mu(y,x)
               \cdot
                   \left(1+(-1)^{\dim(x)}\right)   \\
& = 
       \sum_{y \leq x \leq \ho}
                  (-1)^{\rho(y,x)}
               \cdot
                   \mu(y,x)  \\
& = 
Z([y,\ho])  ,
\end{align*}
where the equality in the second step
follows from the fact the product
$(-1)^{\rho(y,x)} \cdot (-1)^{\dim(x)}$
does not depend
on the element $x$ and thus the term
$\sum_{y \leq x \leq \ho} \mu(y,x)$
vanishes.
Hence the two operators $\eta_{M}$ and $\eta$ agree,
and equation~\eqref{equation_kappa_eta_eta_M}
reduces to
$\Psi(T)^{*} = \varphi\big(\Psi(\PPP)\big) 
  = \omega\big(\av \cdot \Psi(P)\big)$.
\end{proof}

\addtocounter{theorem}{-1}
\addtocounter{section}{-1}
\begin{continuation}
{\rm
The spherical intersection poset $P$ has the flag $\fbar$-vector
$\fbar(\emptyset) = 1$ and
$\fbar(\{1\}) = \fbar(\{2\}) = \fbar(\{1,2\}) = 2$.
Hence the flag $\hhbar$-vector
is given by
$\hhbar(\emptyset) = \hhbar(\{1\}) = \hhbar(\{2\}) = 1$
and
$\hhbar(\{1,2\}) = -1$.
Thus
$\Psi(P) = \maa + \mba + \mab - \mbb$
and
$\Psi(T) = \omega(\av \cdot (\maa + \mba + \mab - \mbb))^{*}
= \mccc + 2 \cdot \mdc$.
}
\end{continuation}
\addtocounter{section}{1}
\addtocounter{theorem}{0}

\begin{example}
{\rm
Let $k$ be a non-negative integer.
Let $c$ and $p$ be positive real numbers such that
$c < p$ and $c + p < \pi/2$. On the $2$-dimensional
unit sphere $S^{2}$ consider the following two closed curves,
which intersect each other in $2 \cdot k$ points:
$$
\phi = c + p \cdot \sin(k \cdot \theta)
\:\:\:\: \text{ and } \:\:\:\:
\phi = -c - p \cdot \sin(k \cdot \theta)  .
$$
Observe that the arrangement of these two curves
on the sphere $S^{2}$ is spherical.
Furthermore, there are
$(2k-1)!! = (2k-1) \cdot (2k-3) \cdots 1$
ways to divide the points into $k$ zero-dimensional spheres.
However, all intersection posets are isomorphic.
The spherical intersection poset has rank $3$ and consists of
$2$ atoms and $k$ coatoms, where each coatom covers
each atom.
The $\ab$-index is given by
$\Psi(P) = (\av+\bv) \cdot (\av+(k-1) \cdot \bv)$.
The $\cd$-index of the induced subdivision of the sphere
is given by
\begin{align*}
\Psi(T)
  & = 
\omega( \av \cdot (\av+\bv) \cdot (\av+(k-1) \cdot \bv) )^{*} \\
  & = 
\omega( \av\av\av )^{*}
  +
(k-1) \cdot \omega( \av\av\bv )^{*}
  +
\omega( \av\bv \cdot (\av+(k-1) \cdot \bv) )^{*} \\
  & = 
\cv^{3}
  +
2(k-1) \cdot \dv\cv
  +
2k \cdot \cv\dv .
\end{align*}
This can be observed directly in the case $k \geq 1$
by noting that the induced subdivision
consists of $2k$ vertices, $2k$ digons and two $2k$-gons.
Note that the calculation also holds for the case $k=0$.
}
\end{example}

\begin{example}
{\rm
Given a complete flag of subspaces
$V_{1} \subseteq V_{2} \subseteqdots V_{n-1}$
in $\Rrr^{n}$ such that
$\dim(V_{i}) = i$,
consider the spherical arrangement in the $(n-1)$-dimensional
sphere $S^{n-1}$ given by
$\{V_{i} \cap S^{n-1}\}_{1 \leq i \leq n-1}$.
The intersection poset is a chain of length $n$,
and its $\ab$-index is $\av^{n-1}$.
The induced stratification consists of two cells
for each dimension and cells of different dimensions
are comparable.
Hence the face poset is the
butterfly poset of rank $n+1$, which has the
$\cd$-index~$\cv^{n}$.
This checks with
$\omega(\av \cdot \av^{n-1}) = \cv^{n}$.
}
\label{example_complete_flag}
\end{example}

Lastly, we observe that the proof of
Theorem~\ref{theorem_spherical}
only used that the Euler characteristic
of each manifold was that of a sphere.
Hence we have the following direct
extension.
\begin{theorem}
Let $M$ be a manifold with Euler characteristic $1$.
Let $\{N_{i}\}_{i=1}^{m}$ be a manifold arrangement
with an intersection poset $P$ such that each
non-empty element $x$ of $P$ satisfies
$\chi(x) = 1 + (-1)^{\dim(x)}$. 
Then the $\cd$-index of the induced Whitney stratification $T$
is given by
\begin{align*}
    \Psi(T)
  & = 
    \omega\big(\av \cdot \Psi(P,\rho,\zeta)\big)^{*}  .
\end{align*}
\end{theorem}

\subsection{Toric arrangements}

We call an affine subspace $V$ in $\Rrr^{n}$
{\em rational} if in the linear system that determines
$V = \{\vec{x} \: : \: A \cdot \vec{x} = \vec{b}\}$
the matrix $A$ has rational entries.
Observe that under the quotient map
$\Rrr^{n} \longrightarrow \Rrr^{n}/\Zzz^{n} = T^{n}$
a rational affine subspace is sent to a subtorus of~$T^{n}$.
Call an affine subspace arrangement $\{V_{i}\}_{i=1}^{m}$ 
in $\Rrr^{n}$ {\em rational} if each subspace is
rational.
A {\em toric arrangement} is the image
of a rational affine subspace arrangement 
under the quotient map.

Define the map $H^{\prime} : \zab \longrightarrow \zab$
by $H^{\prime}(1) = 0$, and 
$H^{\prime}(u \cdot \av) = H^{\prime}(u \cdot \bv) = u$.
Similar to~\cite[Proof of Proposition~8.2]{Billera_Ehrenborg_Readdy}
(see also~\cite[Proposition~3.14]{Ehrenborg_Readdy_Slone})
we have
\begin{equation}
H^{\prime}\big(\Psi(P,\rho,\zeta)\big)
=
\sum_{x \text{ coatom of } P}
\Psi([\hz,x],\rho,\zeta) .
\label{equation_coatom}
\end{equation}

Theorem~\ref{theorem_main} implies the following.

\begin{theorem}
The $\cd$-index induced by a toric arrangement
on the $n$-dimensional torus $T^{n}$, for $n \geq 2$,
is given by
$$
\Psi(T)
=
\frac{1}{2} \cdot
\omega\big(\av \cdot H^{\prime}\big(\Psi(P,\rho,\zeta)\big) \cdot \bv\big)^{*}  ,
$$
where $(P,\rho,\zeta)$ is a quasi-graded
intersection poset of the toric arrangement.
\label{theorem_toric}
\end{theorem}
\begin{proof}
Let $B^{2}$ denote the two-dimensional ball,
that is, a disc. Observe that
the $n$-dimensional torus~$T^{n} = T^{1} \times T^{n-1}$ is the boundary
of $M = B^{2} \times T^{n-1}$.
Also note that
the Euler characteristic is given by $\chi(M) = 0$
for $n \geq 2$. Hence the first term of 
Theorem~\ref{theorem_main} vanishes.
Next, observe that the only elements in the intersection
poset with a non-zero Euler characteristic are the points.
Hence the sum in
Theorem~\ref{theorem_main} reduces to
$$
\Psi(T)
=
          \frac{1}{2} \cdot
      \sum_{x \text{ coatom of } P}
          \omega\big(\av \cdot \Psi([\hz,x]) \cdot \bv\big)^{*}  ,
$$
that is, a sum over all coatoms in the intersection poset $P$.
Since $w \longmapsto \omega(\av \cdot w \cdot \bv)$ is linear,
by equation~\eqref{equation_coatom} the sum reduces further
to the statement of the theorem.
\end{proof}

\begin{remark}
{\rm
Theorem~\ref{theorem_toric}
strengthens~\cite[Theorem~3.12]{Ehrenborg_Readdy_Slone}
in a number of directions.
First,
the previous result was only proved for an arrangement
of codimension one tori, that is, the image of a rational
(affine) hyperplane arrangement.
This is no longer the case.
Secondly,~\cite{Ehrenborg_Readdy_Slone} required that
the toric arrangement induced
a regular subdivision of the torus.
Again, we no longer require the regularity condition.
Lastly, the previous result computes the $\ab$-index of
the face poset.
This is not an Eulerian poset. However, it is almost
an Eulerian quasi-graded poset with the weighted zeta function
given by the classical zeta function~$\zeta$.
We only have to change the value of 
$\zetabar(\hz,\ho)$ to $\chi(M) = 0$ to
make it Eulerian.
This explains why the earlier result
in~\cite{Ehrenborg_Readdy_Slone}
had an extra~$(\av-\bv)^{n+1}$ term.
}
\end{remark}

Similar to Example~\ref{example_complete_flag},
we have a toric analogue.
\begin{example}
{\rm
Let
$V_{0} \subseteq V_{1} \subseteqdots V_{n-1}$
be a complete flag of rational affine subspaces
in~$\Rrr^{n}$ such that $\dim(V_{i}) = i$.
Consider the toric arrangement 
obtained by the quotient map.
The intersection poset is a chain of length $n+1$
whose $\ab$-index is~$\av^{n}$. 
Hence the $\cd$-index of the face poset is 
given by
$1/2 \cdot \omega(\av \cdot \av^{n-1} \cdot \bv)^{*} = \dv \cv^{n-1}$.
}
\label{example_toric_complete_flag}
\end{example}

\section{Concluding remarks}

The next natural step is to ask for inequalities for
the $\cd$-coefficients of arrangements.
For the classical case of central hyperplane arrangements,
equivalently for zonotopes, the $\cd$-index is minimized
on the $n$-dimensional cube;
see~\cite[Corollaries~7.5 and~7.6]{Billera_Ehrenborg_Readdy}.
This was continued by Nyman and Swartz~\cite{Nyman_Swartz}
in special cases. Ehrenborg's lifting technique
for polytopes~\cite{Ehrenborg_lifting}
was extended to zonotopes; see~\cite{Ehrenborg_zonotopes}.
However, the question 
of determining inequalities
is wide open for general manifold arrangements.
As a first step one should consider 
arrangements with codimension one submanifolds.

For results concerning inequalities
and unimodality of the $f$-vector
of zonotopes, see~\cite{Fukuda_Saito_Tamura,Fukuda_Saito_Tamura_Tokuyama,Fukuda_Tamura_Tokuyama}.
The minimal and maximal
number of connected components of arrangements
has been studied
by Shnurnikov~\cite{Shnurnikov}
in the cases of Euclidean, projective
and Lobachevski{\u{\i}} spaces.
Moci~\cite{Moci} introduced a generalized Tutte polynomial
for spherical and toric arrangements. It would be interesting
to develop a similar polynomial for manifold arrangements.
Is there a natural subclass of manifold arrangements where
this would be successful?

\section*{Acknowledgements}

The authors thank Mark Goresky for many helpful
discussions
and the referees for their careful comments.
The first author also thanks the Institute for Advanced Study
for a productive research visit in May 2012.
The first author was partially supported by
National Science Foundation grant DMS 0902063.
This work was partially supported by a grant from the Simons Foundation (\#206001 to Margaret~Readdy).

\newcommand{\journal}[6]{{\sc #1,} #2, {\it #3} {\bf #4} (#5), #6.}
\newcommand{\preprint}[3]{{\sc #1,} #2, preprint~#3.}
\newcommand{\book}[4]{{\sc #1,} #2, #3, #4.}

\end{document}